\documentclass[12pt]{amsart}
\usepackage{amssymb,amsmath,amsthm}
\numberwithin{equation}{section}
\begin{document}

\theoremstyle{plain}
\newtheorem{theorem}{Theorem}[section]
\newtheorem{lemma}[theorem]{Lemma}
\newtheorem{proposition}[theorem]{Proposition}
\newtheorem{corollary}[theorem]{Corollary}
\newtheorem{conjecture}[theorem]{Conjecture}

\def\mod#1{{\ifmmode\text{\rm\ (mod~$#1$)}
\else\discretionary{}{}{\hbox{ }}\rm(mod~$#1$)\fi}}

\theoremstyle{definition}
\newtheorem*{definition}{Definition}

\theoremstyle{remark}
\newtheorem{remark}{Remark}[section]
\newtheorem{example}{Example}[section]
\newtheorem*{remarks}{Remarks}
\newcommand{\ndiv}{\hspace{-4pt}\not|\hspace{2pt}}
\newcommand{\cc}{{\mathbb C}}
\newcommand{\qq}{{\mathbb Q}}
\newcommand{\rr}{{\mathbb R}}
\newcommand{\nn}{{\mathbb N}}
\newcommand{\zz}{{\mathbb Z}}
\newcommand{\pp}{{\mathbb P}}
\newcommand{\al}{\alpha}
\newcommand{\be}{\beta}
\newcommand{\ga}{\gamma}
\newcommand{\ze}{\zeta}
\newcommand{\om}{\omega}
\newcommand{\mz}{{\mathcal Z}}
\newcommand{\mi}{{\mathcal I}}
\newcommand{\ep}{\epsilon}
\newcommand{\la}{\lambda}
\newcommand{\de}{\delta}
\newcommand{\tha}{\theta}
\newcommand{\De}{\Delta}
\newcommand{\Ga}{\Gamma}
\newcommand{\si}{\sigma}
\newcommand{\Exp}{{\rm Exp}}
\newcommand{\legen}[2]{\genfrac{(}{)}{}{}{#1}{#2}}
\def\End{{\rm End}}
\title{Linearly dependent powers of binary quadratic forms}

\author{Bruce Reznick}
\address{Department of Mathematics, University of 
Illinois at Urbana-Champaign, Urbana, IL 61801} 
\email{reznick@math.uiuc.edu}
\date{\today} 

\begin{abstract}
Given an integer $d \ge 2$, what is the least $r$ so that
there is a set of binary quadratic forms $\{f_1,\dots,f_r\}$ for which
$\{f_j^d\}$ is non-trivially linearly dependent? We show that if $r \le 4$,
then $d \le 5$, and for $d \ge 4$, construct such a set with $r = \lfloor d/2\rfloor + 2$. 
Many explicit examples are given, along with techniques
for producing others.
\end{abstract}

\thanks{The author was supported by
Simons Collaboration Grant 280987.}

\maketitle

\section{Introduction}

For a fixed positive integer $k$, let $H_k(\cc^2)$ denote the
$(k+1)$-dimensional vector space of binary forms of degree $k$ with
complex coefficients. We say that two such forms are {\it distinct} if
they are not proportional, and we say that a set $\mathcal F = \{f_1,\dots,f_r\} \subset 
H_k(\cc^2)$ is {\it honest} if its elements are pairwise distinct. For $d \in \nn$, 
let $\mathcal F^d = \{f_1^d,\dots,f_r^d\}$; if $\mathcal 
F$ is honest, then so is  $\mathcal F^d$. 

When $k=1$, there is a simple classical criterion for the linear dependence
of  $\mathcal F^d$; see, e.g. \cite[Thm.4.2]{Re2}.

\begin{theorem}\label{T1}
If $\mathcal F = \{f_1,\dots,f_r\} \subset H_1(\cc^2)$ is honest, 
then $\mathcal F^d = \{f_1^d,\dots,f_r^d\}$ is linearly independent if and
only if $r \le d+1$. 
\end{theorem}
A version of this criterion is  generally true for $k \ge 2$; see, e.g. \cite[Thm.1.8]{Re4}. 
(The proofs of these theorems are given at the start of section two.)
 \begin{theorem}\label{T2}
If $\mathcal F = \{f_1,\dots,f_r\} \subset H_k(\cc^2)$, then it is {\it generally} true that  
$\mathcal F^d$ is linearly independent if and only if $r \le kd+1$. 
\end{theorem}

But there are singular cases, and these will be the focus of this paper. 
It is easy to find smaller values of $r$ for which
$\mathcal F^d$ is linearly dependent; for example,
  the Pythagorean parameterization gives three
quadratics whose squares are dependent:
 \begin{equation}\label{E:pyt}
(x^2-y^2)^2 +  (2xy)^2 = (x^2 + y^2)^2.
\end{equation}
 There are other ways of finding small dependent sets:
let $\{g_j(x,y)\}$ be an honest set of $d+2$ linear forms, then both $\{g_j(x^k,y^k)\}$
and  $\{\ell(x,y)^{k-1}g_j(x,y)\}$ (for a fixed linear form $\ell$) will be
dependent sets in  $H_k(\cc^2)$.

Given $r, d \in \mathbb N$, we say that an honest set of forms
$\{f_1,\dots,f_r\} \subseteq H_k(\cc^2)$ is a {\it $\mathcal W_k(r,d)$-set} if $\{f_j^d\}$ is linearly 
dependent. For example, \eqref{E:pyt} presents the $\mathcal W_2(3,2)$ set 
$\{x^2-y^2, 2xy, x^2+y^2\}$.  Let $\Phi_k(d)$ denote the smallest $r$ for which
 a $\mathcal W_k(r,d)$ set
 exists; clearly, $\Phi_k(d) \ge 3$. Theorem \ref{T1} implies that $\Phi_1(d) = d+2$.
  
Our goal in this paper is two-fold. First, we give upper and lower bounds for $\Phi_k(d)$ for
$k \ge 2$. 
Second, we describe all $\mathcal W_2(\Phi_2(d),d)$ sets for $d \le 5$. 
In (5) and (6) below, we use a
peculiar-looking function. If $e \ | \ d$, let
\[
\Theta_e(d) := 1 + \min_{t \in \nn} (t\cdot \tfrac de + \lfloor \tfrac et \rfloor).
\]
 We summarize our main results.
\begin{theorem}[Main Theorem]\label{MT}
\ 
\begin{enumerate}
\item $\Phi_{k+1}(d) \le \Phi_k(d)$.
\item $\Phi_k(3) = 3$.
\item (Liouville) $\Phi_k(d) \ge 4$  for $d \ge 3$ and all $k$.
\item  (Hayman) $\Phi_k(d) > 1 + \sqrt{d+1}$ for $d \ge 3$ and all $k$.
\item (Molluzzo-Newman-Slater) $\Phi_d(d) \le \Theta_d(d) 
 = 1 + \lfloor \sqrt{4d+1} \rfloor$.
\item If $e \ | \ d$, then  $\Phi_e(d) \le  \min\{\Theta_k(d): k \ge e,\  k \ | \ d\}$.
\item $\Phi_k(d) = 4$ for $d = 3,4,5$ and $k \ge 2$.
\item $\Phi_2(d) \ge 5$ for $d \ge 6$.
\item $\Phi_2(d) = 5$ for $d = 6,7$.
\item $\Phi_2(14) \le 6$.
\item $\Phi_2(d) \le \lfloor d/2 \rfloor +2$ for $d \ge 4$.
\end{enumerate}
\end{theorem}

All new parts of the Main Theorem except (8) and (11) have short proofs; these are given in 
section two. Examples give upper bounds for $\Phi_k(d)$; lower bounds are harder 
to find. The anomalous value in (10) for $d=14$ is difficult to explain, and prevents us
from conjecturing (11) as the exact value.
This problem has been studied in \cite{GH} and \cite{NS} without the degree
condition on the summands. The recent \cite{NSS} contains a
generalization of this question, replacing
$f_i^d$ with $\prod_j f_{ij}^{a_j}$ for fixed tuples $(a_j)$.

If $\mathcal F$ is a $\mathcal W_k(r,d)$ set, then there is an obvious way to transform the linear
dependence of the $d$-th powers into a more natural expression for any $m, 1 \le m \le r-1$:
\begin{equation}\label{E:scal}
\sum_{j=1}^r \la_j f_j^d = 0 \quad (\la_j \neq 0) \quad \implies p = \sum_{j=1}^m \tilde f_j^d = 
\sum_{j=m+1}^r \tilde f_j^d,
\end{equation}
where $\tilde f_j = (\pm \la_j)^{1/d} f_j$, for some $p$. In particular, a $\mathcal W_k(2m,d)$ set 
addresses the classical question of parameterizing two equal sums of $m$ $d$-th powers.
In this case, 
we say that \eqref{E:scal} gives {\it two representations} of $p$ as a sum of $m$ $d$-th powers.

If $\al x + \be y$ and $\ga x + \de y$ are distinct, then the map
$M:=(x,y) \mapsto  (\al x + \be y, \ga x + \de y)$
is an invertible change of variables (or {\it linear change} for short); let $(f \circ M)(x,y)$ 
denote $f(\al x + \be y, \ga x + \de y)$. (This is a {\it scaling} if $\be=\ga=0$.) If all members of
 $\mathcal F$ are subject to the same linear 
change, then the linear dependence of their $d$-th powers is unaltered. 
Any $\mathcal W_k(r,d)$ set can have its elements permuted and multiplied by various
non-zero constants without essentially affecting the nature of the dependence. 

So, suppose $\mathcal F$ is a $\mathcal W_k(r,d)$ set and
\begin{equation}\label{E:sim1}
\sum_{j=1}^r \la_j f_j^d = 0.
\end{equation}
If $\pi \in S_r$ is a permutation of $\{1,\dots r\}$, $c = (c_1,\dots, c_r) \in \mathbb (\cc\setminus 
\{0\})^r$, $M$ is a linear change, and $g_j = c_j(f_{\pi(j)} \circ M)$, $1 \le j \le r$,
then \eqref{E:sim1} is equivalent to 
\begin{equation}\label{E:sim2}
\sum_{j=1}^r (\la_{\pi(j)}\cdot c_j^{-d}) g_j^d= 0.
\end{equation}

In this situation, we say that  $\mathcal F = \{f_j\}$ and $\mathcal G = \{g_j\}$ 
(and the corresponding identities $\eqref{E:sim1} ,\eqref{E:sim2}$) are {\it cousins}. 
It is easy to  show cousinhood by exhibiting $M$, $\pi$ and $c$. Proving that 
$\mathcal F$ and $\mathcal G$
are {\it not} cousins may require {\it ad hoc} arguments.

We aim to  present identities as symmetrically as possible, often guided by 
 an old idea of Felix Klein.  Associate to each non-zero linear form 
$\ell(x,y) = s x - t y$ the image of $t/s \in \cc^*$ on the unit sphere $S^2$ under the
Riemann map. (Assign $\ell(x,y) = y$ to $\infty$ and $(0,0,1)$.)
Then associate to the binary form $\phi(x,y)= \prod_{j=1}^k (s_j x - t_j y)$
the image under the Riemann map of $\{t_j/s_j\}$,  and call it the {\it Klein set of $\phi$}. 
Given \eqref{E:sim1}, we shall be interested in the Klein set of  $\prod_{j=1}^r f_j$.   
In \eqref{E:pyt}, the Klein set of $(x^2-y^2)(2xy)(x^2+y^2)$ 
is the regular octahedron with vertices $\{\pm e_k \}$. 

Under the linear change $M$$:(x,y) \mapsto  (\al x + \be y, \ga x + \de y)$, the root
$t/s \mapsto T(t/s)$, where $T$ is the M\"obius transformation 
$T(z) = \frac{\de z - \be}{-\ga z+\al}$. 
Every rotation of the sphere corresponds to a M\"obius transformation of the complex plane, and so 
a rotation of the Klein set can be effected by  imposing a linear change on the forms. 
(Unfortunately, not every M\"obius transformation gives a rotation.) It often happens that 
$p = \sum f_j^d$ and $ p = p \circ M$, but $\sum (f_j\circ M)^d$ gives a different 
representation for $p$.

A trivial remark is surprisingly useful: 
\[
p = f_1^d + f_2^d = f_3^d + f_4^d \implies q = f_1^d -f_3^d = f_4^d - f_2^d
\]
for suitable forms $p,q$; we call this a {\it flip}. For $k=2$ and $d=3,4$, it can happen that $q$ has 
a third representation as $q = f_5^d + f_6^d$, but that no such new expression exists for $p$.
If
$f_1^d + f_2^d = f_3^d + f_4^d$ and $g_1^d + g_2^d = g_3^d + g_4^d
= g_5^d + g_6^d $, then $\mathcal F = \{f_1,\dots, f_4\}$ is a 
cousin of $\mathcal G = \{g_1,\dots,g_4\}$ and we say that  $\mathcal F$ is a 
{\it sub-cousin} of $\mathcal G' = \{g_1,\dots,g_6\}$.

We now present some examples of small dependent sets of $d$-th powers.
For integer $m \in \nn$, let $\ze_m =  e^{\frac{2\pi i}m}$ be a primitive $m$-th root of unity, with the 
usual conventions that $\om = \ze_3$ and $i = \ze_4$. A few interesting Klein sets will be noted.

The cubic identity with the simplest coefficients is probably
\begin{equation}\label{E:cubic}
\begin{gathered}
(x^2 + x y - y^2)^3 + (x^2 - x y - y^2)^3  = 2(x^2)^3 + 2(-y^2)^3 = 2 x^6 - 2y^6. 
\end{gathered}
\end{equation}
The right-hand side of \eqref{E:cubic} is unchanged by the scalings $y \to \om y$ and $y \to \om^2 
y$, so \eqref{E:cubic} shows that $2x^6-2y^6$ is a sum of two cubes in four  different ways. 
Under the linear change $(x,y) \mapsto (\al + \be, \al -\be)$, \eqref{E:cubic} is due to G\'erardin
see \cite[p.562]{D} in 1910; in its present form, it was noted by Elkies in \cite[p.542]{DG}. 

Here are two very simple quartic identities. 
The first generalizes to higher even degree; see \eqref{E:monom}, and the second is
in $\mathbb Z[x,y]$:
\begin{equation}\label{E:quartic}
(x^2 + y^2)^4 + (\om x^2 + \om^2 y^2)^4 + (\om^2 x^2 + \om y^2)^4= 18(xy)^4.
\end{equation}
\begin{equation}\label{E:quarcous}
(x^2 + 2xy)^4 + (2xy + y^2)^4 + (x^2 - y^2)^4 = 2(x^2+xy+y^2)^4.
\end{equation}
These are cousins. Upon making the linear change $(x,y) \mapsto  (i(x - \om y), (x - \om^2y))$ 
and division by $\sqrt{-3}$, \eqref{E:quartic} becomes \eqref{E:quarcous} up to a permutation of 
terms. The Klein
set of \eqref{E:quartic} is a regular hexagon at the  equator plus the poles.

A remarkable identity for $d=5$ was discovered independently by A. H. 
Desboves in 1880 (see  \cite{Deb},  \cite[p.684]{D}) and N. Elkies in 1995 (see \cite[p.542]{DG}):
\begin{equation}\label{E:quintic11}
\sum_{k=0}^3 (-1)^k (i^k x^2 + \sqrt{-2} \ xy + i^{-k}y^2)^5 = 0.
\end{equation}
The Klein set of \eqref{E:quintic11} is the  cube with vertices 
$\{(\pm\frac{\sqrt2}{\sqrt3},0,\pm\frac 1{\sqrt3} ), (0,\pm\frac{\sqrt2}{\sqrt3},\pm\frac 1{\sqrt3} )\}$.

The next two examples appear to be new in detail, but are in the spirit of \cite[\S4]{Re1}; 
the third explicitly appears there as (4.15); each is  derived in section two:

\begin{equation}\label{E:six}
\begin{gathered}
\sum_{k=0}^3 i^k(x^2 + i^k y^2)^6 = 80(x y)^6,
\end{gathered}
\end{equation}

\begin{equation}\label{E:seven}
\begin{gathered}
\sum_{k=0}^3 \left( i^{-k} x^2 +\sqrt{-6/5}\ x y + i^k y^2\right)^7 =
26\sqrt{3}  \cdot (-\sqrt{8/5}\ xy)^7,
\end{gathered}
\end{equation}

\begin{equation}\label{E:14}
\sum_{j=0}^4(\ze_5^j x^2 + i x y + \ze_5^{-j}y^2)^{14}
= 5^7(xy)^{14}. 
\end{equation}
The Klein set of \eqref{E:14} is the regular icosahedron, oriented so the vertices
are the two poles plus two parallel regular pentagons at latitude $z = \pm \frac 1{\sqrt 5}$.

The second main focus of this paper is the characterization of $\mathcal W_2(\Phi_2(d),d)$ 
sets for $d=3,4,5$. The characterization of $\mathcal W_k(3,2)$ is classical,
and can be proved by emulating the standard analysis of $a^2 + b^2 = c^2$ over $\nn$.
\begin{theorem}\label{W2}
If $p,q,r \in \mathbb C[x_1,\dots,x_n]$, $n \ge 1$ and 
$p^2 + q^2 = r^2$, then 
then there exist $f,g,h \in \mathbb C[x_1,\dots,x_n]$ so that 
$p = f(g^2 - h^2),  q = f(2gh),  r = f(g^2 + h^2).$
\end{theorem}

The proof of the following theorem will be found in the companion paper \cite{Re3}.
\begin{theorem}\label{W3}
Every $\mathcal W_2(4,3)$ set is a sub-cousin of a member of
the $\mathcal W_2(6,3)$ family given below, for some $\al \neq 0, \pm 1$:
\begin{equation}\label{E:threefold}
\begin{gathered}
(\al x^2 - xy + \al y^2)^3 +\al( - x^2 +\al xy - y^2)^3 \\ 
= (\om^2 \al x^2 - xy + \om \al y^2)^3 +\al( - \om^2 x^2 +\al xy - \om y^2)^3 \\
= (\om\al x^2 - xy + \om^2 \al y^2)^3 +\al( - \om x^2 +\al xy - \om^2 y^2)^3 
 \\ =  (\al^2-1)(\al x^3 + y^3)(x^3 + \al y^3).
\end{gathered}
\end{equation}
If the first two lines of \eqref{E:threefold} are read as $f_1^3 + f_2^3 = f_3^3 + f_4^3$, then
$f_1^3 - f_4^3 = f_3^3  - f_2^3$ also has a third representation as a sum of two cubes, but
$f_1^3 - f_3^3 = f_4^3  - f_2^3$ does not. 
\end{theorem}

(Put $(\al,x,y) \mapsto (i,\ze_8^3x,\ze_8^5y)$  in the first line of
\eqref{E:threefold} to get  \eqref{E:cubic}.)
After the linear change: $(x,y) \mapsto (i x + \sqrt{3}\ y, i x-\sqrt{3}\ y)$,
 \eqref{E:threefold} becomes
\begin{equation}\label{E:threefoldreal}
\begin{gathered}
((1 - 2 \al) x^2 + 3 (1 + 2 \al)y^2)^3 + \al ((2-\al) x^2 - 3 (2 + \al) y^2)^3 \\ 
= ((1 + \al) x^2 + 6 \al x y + 3 (1-\al) y^2)^3 + \al ( -(1 + \al) x^2 - 6 x y + 3 (1 -\al) y^2)^3\\
= ((1 + \al) x^2 - 6 \al x y + 3 (1-\al) y^2)^3 + \al( -(1 + \al) x^2 + 6 x y + 3 (1 -\al) y^2)^3.
\end{gathered}
\end{equation}
If $\al \in \mathbb Q$, then all forms in \eqref{E:threefoldreal} are in $\mathbb Q[x,y]$, and if $\al$ 
is a rational cube, then \eqref{E:threefoldreal} gives  
solutions to $f_1^3 + f_2^3 = f_3^3 + f_4^3$ in $\mathbb Q[x,y]$. Historically, these
were used to parameterize solutions to the Diophantine equations $a^3+b^3=c^3+d^3$ over $\nn$.

\begin{theorem}\label{W4}
Every $\mathcal W_2(4,4)$ is a cousin of  \eqref{E:quartic} or a sub-cousin of
 \eqref{E:quartic13}:
\begin{equation}\label{E:quartic13}
\begin{gathered}
(x^2 + \sqrt 3\ xy - y^2)^4 - (x^2 - \sqrt 3\ xy - y^2)^4 \\
= (\om^2 x^2 + \sqrt 3\ x y - \om y^2)^4 - (\om^2 x^2 - \sqrt 3\ x y -
\om y^2)^4 \\
= (\om\ x^2 + \sqrt 3\ x y - \om^2 y^2)^4 - (\om\ x^2 - \sqrt 3\ x y -
\om^2 y^2)^4 \\
= 8 \sqrt 3\  xy\ (x^6 - y^6). 
\end{gathered}
\end{equation}
\end{theorem}

In an earlier version of this work (see e.g. \cite[(3.9)]{Re1}), the identity
\begin{equation}\label{E:quartic12}
\begin{gathered}
(\sqrt 3\ x^2 + \sqrt 2\ xy - \sqrt3\ y^2)^4 +
(\sqrt 3\ x^2 - \sqrt 2\ xy - \sqrt3\ y^2)^4 \\
= (\sqrt 3\ x^2 + i \sqrt 2\ xy + \sqrt3\ y^2)^4 +  
(\sqrt 3\ x^2 - i \sqrt 2\ xy + \sqrt3\ y^2)^4 \\
= 18x^8 - 28 x^4y^4 + 18 y^8.
\end{gathered}
\end{equation}
was given as an alternative in Theorem \ref{W4}; 
\eqref{E:quartic12} turns out to be a sub-cousin of \eqref{E:quartic13}, see Theorem \ref{T:44}.
When scaled, \eqref{E:quartic12} appears in Desboves \cite[p.243]{Deb}.
The set in \eqref{E:quartic} is not a sub-cousin of \eqref{E:quartic13}:
three of the quadratics in \eqref{E:quartic} are linearly dependent, and 
 no three quadratics in \eqref{E:quartic13} are dependent.

The situation for quintics is simpler.
\begin{theorem}\label{W5}
Every $\mathcal W_2(4,5)$ set is a cousin of \eqref{E:quintic11}.
\end{theorem}

Here is an outline of the rest of the paper.
In section two, we prove Theorems \ref{T1} and \ref{T2} and Theorem \ref{MT} except (8).
We also recall ``synching" from \cite{Re1} 
as a tool for  finding ``good" $\mathcal W_k(r,d)'s$ -- the idea was inspired by a formula
of Molluzzo \cite{M} -- and use it to prove several parts of Theorem \ref{MT}.

In section three, we recall two results familiar to 19th century algebraists:  
a specialization of Sylvester's algorithm for determining the sums of two $d$-th powers of linear
forms and a result on the simultaneous  diagonalization of quadratic forms. 
We use these to lay out our strategy for proving Theorem \ref{MT}(8). Suppose
\[\label{E:key4}
p(x,y) = f_1^d(x,y) + f_2^d(x,y) = f_3^d(x,y) + f_4^d(x,y) 
\]
for an honest set $\{f_1,f_2,f_3,f_4\}$ of quadratics. There is a linear 
change which simultaneously diagonalizes $f_1$ and 
$f_2$ (making $p$ even), but neither $f_3$ nor $f_4$ is even. We then make a 
systematic study of non-even $\{f_3,f_4\}$ for which $p =f_3^d + f_4^d$ is even, and check back to
see whether $p$ can be written as $f_1^d+f_2^d$. For $d \ge 3$, a shorter method can be used 
to prove Theorem \ref{W3}; see  the companion paper  \cite{Re3}.
 
Section four is devoted to  implementing
in detail the strategy outlined above; this simultaneously proves Theorems  
\ref{W4} and \ref{W5}, as well as Theorem \ref{MT}(8).
The proofs of Theorems \ref{T:tameanswer} and \ref{T:wildanswer} contain a great deal of 
``equation wrangling";
however, the reader should know that this has been greatly condensed from earlier drafts. 

In section five, we do a brief review of the literature in the subject and derive the examples for 
for $d \le 5$ via {\it a priori} constructions.  We also discuss how Newton's theorem on
symmetric forms helps explain \eqref{E:14}, similar to the argument for \eqref{E:quintic11} 
given in \cite{Re1}. Corollaries \ref{C:62} and  \ref{C:63} present the classification of forms which 
can be written as a sum of two $d$-th powers of quadratic forms and, for $d \ge 4$,
 those which have more than one representation. We  suggest some further areas of exploration 
 and finish with Conjecture \ref{C:end} about the true growth of $\Phi_k(d)$.

 The author has been working on this project for a very long time; online seminar notes 
\cite{Re5}
are dated 2000. He wishes to thank 
Andrew Bremner, Noam Elkies, Jordan Ellenberg, Andrew Granville, Samuel Lundqvist,
Cordian Rainer and Boris Shapiro for encouraging conversations and
useful emails, even if after all this time, they don't remember what they said. 
Many thanks to Becky  Burner of the Illinois Mathematics Library for finding an online copy of
\cite{Deb}.
 
\section{Some proofs, and synching} 
We begin with proofs of Theorems \ref{T1} and \ref{T2}.

\begin{proof}[Proof  of Theorem \ref{T1}]
If $r > d+1 = \dim(H_d(\cc^2)$), then   $\mathcal F^d$ is dependent. Suppose $ r \le d+1$ 
and let $f_i(x,y) = \al_i x + \be_i y$. Define (if necessary) distinct $f_j$ for $r+1 \le j \le d+1$
by $(\al_j,\be_j) = (1,m_j)$, where $m_j\al_i \neq \be_i, 1 \le i \le r$, and express 
$\{f_1^d,\dots,f_{d+1}^d\}$  in terms of the  basis $\{\binom dv x^{d-v}y^v\}$. The resulting
$(d+1) \times (d+1)$ matrix, $[\al_i^{d-v}\be_i^v]$, is Vandermonde, with determinant
$
\prod_{1 \le i < j \le d+1} (\al_i \be_j  - \al_j \be_i) \neq 0,
$
since $\mathcal F$ is honest. 
\end{proof}
\begin{proof}[Proof  of Theorem \ref{T2}]
Again, if $r > kd+1$, then  $\mathcal F^d$ is linearly dependent
by dimension. Suppose $f_j(x,y) = \sum_{\ell=0}^k \binom k{\ell}
\al_{\ell,j}x^{k-\ell}y^{\ell}$. If $r < kd+1$, again add
pairwise distinct elements and assume that $r = kd+1$.  Express $\{f_j^d\}
$ in terms of the monomial basis $\{\binom {kd}v
x^{kd-v}y^v\}$, obtaining a square matrix of order $kd+1$ whose
entries are polynomials in the variables  $\{\al_{\ell,j}\}$, and whose determinant is
a polynomial $P(\{\al_{\ell,j}\})$. If we 
specialize to $f_j(x,y) = (x + jy)^k$, $1 \le j \le kd+1$, then $\al_{\ell,j} = j^{\ell}$, and 
$\mathcal F^d = \mathcal G^{kd}$ for $\mathcal G = \{x + j y\}$. By Theorem \ref{T1},
$\mathcal G^{kd}$ is linearly independent, hence $P(\{j^\ell\})\neq 0$, and so $P$ is not identically 
zero.  That is, $\mathcal F^d$, generally, is  linearly independent.
\end{proof}

We defer the proofs of Theorem \ref{MT}(5), (6) and (11) until we have
defined synching; (8) will require sections three and four.
 \begin{proof}[Partial Proof  of Theorem \ref{MT}]
 \ 
 
 (1) If $g_j(x,y) = xf_j(x,y)$, then $\sum \la_j f_j^d = 0 \implies \sum \la_j g_j^d = 0$.
 
(2)  This follows from  \eqref{E:pyt} and (1).

(3) As noted in  \eqref{E:scal}, the existence of a $\mathcal W_k(3,d)$ set for $d \ge 3$ would 
imply the existence of a nontrivial identity
\[
f_1^d(x,y) + f_2^d(x,y) = f_3^d(x,y).
\]
After a linear change, we may assume that $f_j(x,y)$ is not a multiple of 
$y^k$. Let $p_j(t) = f_j(t,1)$. Then $p_1^d(t) + p_2^d(t) = p_3^d(t)$, where the $p_j$'s 
are non-constant polynomials. In 1879, Liouville proved that the Fermat equation 
$X^d+Y^d=Z^d$ has  no non-constant solutions over $\mathbb C[t]$ for $d \ge 3$. 
(See  \cite[pp.263--265]{Rib} for a proof.)

(4) More generally, the elements of any $\mathcal W_k(r,d)$ set can be scaled as in 
\eqref{E:scal} so that  $\sum_{j=1}^{r-1} f_j^d(x,y) = f_r^d(x,y)$. Once again, by letting 
$p_j(t) = f_j(t,1)$ and $q_j(t) = f_j(t)/f_r(t)$  we obtain a set of $r-1$ rational functions so that 
$\sum_{j=1}^{r-1} q_j^d(t) = 1$.  A 1984 theorem of Hayman
\cite{H1} says that if $\{\phi_j\}$, $1 \le j \le r-1$, are $r-1$ holomorphic
functions in $n$ complex variables, no two of which are proportional,
and $\sum_{j=1}^{r-1} \phi_j^d = 1$, then $d <  (r-1)^2-1$, so $r > 1 + \sqrt{d+1}$. This
was culmination of the work of Green \cite{G} and others; see \cite[pp.438-440]{GH}
for a clear exposition and history.

(7) The equality for $k=2$ follows from combining (3) with the equations  \eqref{E:cubic}, 
\eqref{E:quartic} and \eqref{E:quintic11}; for $k \ge 3$, apply (1).

(9) Subject to the as-yet unproved (8), this follows from \eqref{E:six} and \eqref{E:seven}.

(10) This follows from  \eqref{E:14}.
\end{proof}

Recall that for an integer $m \ge 2$ and for $s \in \mathbb Z$,
\begin{equation}\label{E:orthog}
\frac 1m \sum_{j=0}^{m-1}\ze_m^{sj} = \begin{cases} 0 & \text{if } m \ \nmid \ s, 
   \\ 1 & \text{if } m\ \mid \ s. \end{cases}
\end{equation}
Synching was introduced in  \cite[\S4]{Re1} and is a generalization of the familiar formulas in which
$\frac12 (f(x,y) \pm f(x,-y))$ give the even and odd parts of $f$.
\begin{theorem}\label{T:synching}
Suppose  $p(x,y) = \sum_{i=0}^k a_i x^{k-i}y^i \in H_k(\cc^2)$.
Then
\begin{equation}\label{E:fullsynch}
 \frac 1m\sum_{j=0}^{m-1} \ze_m^{-rj}p(x,\zeta_m^j y) = 
 \sum_{\substack{i \equiv r \mod m,\\0 \le i \le k}} a_i x^{k-i}y^i.
\end{equation}
\end{theorem}
\begin{proof}
We expand the left-hand side of \eqref{E:fullsynch}, switch the order of summation:
\[
\frac 1m\sum_{j=0}^{m-1} \ze_m^{-rj}p(x,\zeta_m^j \ y) = 
 \sum_{i=0}^k \left(\frac 1m\sum_{j=0}^{m-1} \ze_m^{-rj}\ze_m^{ij} \right)a_i x^{k-i}y^i,
\]
and then apply \eqref{E:orthog} to the inner sum of $\ze_m^{(i-r)j}$.
\end{proof}

In our applications, $p=f^d$; for example, if $p(x,y) = (x + \al y)^d$, then: 
\begin{equation}\label{E:basicsynch}
\frac 1m\sum_{j=0}^{m-1} \ze_m^{-rj} ( x + \ze_{m}^j\al y)^d = 
\sum_{-\frac rm \le i \le\frac{d-r}m} \binom d{r+im} \al^{r+im}x^{d-r-im}y^{r+im}. 
\end{equation}

 \begin{proof}[Proof of Theorem \ref{MT}(5), (6)]
 We generalize an identity found in Molluzzo's thesis \cite{M} (with
 $\ell=d$) and discussed in \cite[p.485]{NS}; it follows from  \eqref{E:basicsynch} with $r=0$ that 
  \begin{equation}\label{E:Moll}
\sum_{j=0}^{m-1}(x^\ell + \zeta_m^j y^\ell)^d = m \sum_{i=0}^{\lfloor d/m \rfloor} \binom d{im}
x^{\ell d - im\ell}y^{im\ell}.
\end{equation}

Suppose now that $d = ee'$, $\ell = e$ and $m = te'$ is a multiple
of $e'$.  Then the left-hand side of \eqref{E:Moll} is a sum of $m$ $d$-th powers, and since 
$d \ | \ im\ell = itd$, the right-hand side is a sum of $1+\lfloor d/m \rfloor$ $d$-th powers. Thus
the total number of summands is $1 + t\cdot \frac de   + \lfloor \frac et \rfloor$. We choose
$t$ to minimize this sum, obtaining $\Theta_e(d)$. 

Newman and Slater  took $d=e$, so $e' = 1$  (\cite[p.485]{NS}); the
minimum in $\Theta_d(d)$ is found by choosing  
$m \in \{\lfloor \sqrt{d} \rfloor, 1+\lfloor \sqrt{d} \rfloor\}$, giving
$\Phi_d(d) = 1 + \lfloor \sqrt{4d+1} \rfloor$. 

If $e < d$, then $\Theta_e(d)$ is generally larger than $\Theta_d(d)$, 
since some $m$'s are skipped in computing the minimum; however, $\Theta_e(d)$ need not
be monotone in $e$, so Theorem \ref{MT}(1) need not be  be implemented. 
\end{proof}
The first instance of non-monotonicity  in $\Theta_e(d)$ occurs at $d=72$; in general,  
 $\Theta_{8n}(72n^2) = \Theta_{9n}(72n^2) = 1+17n$, but $\Theta_{12n}(72n^2) = 1+18n$. This
 suggests interesting questions in combinatorial number theory which we hope to
 pursue elsewhere. 

 When $d$ is even,  a specialization of   \eqref{E:basicsynch} can be made more 
 symmetric:
\begin{corollary}\label{centralsync}
\begin{equation}\label{E:monom8}
\frac 1{s+1}\cdot \sum_{j=0}^{s} (\ze_{2s+2}^{-j} x + \ze_{2s+2}^{j}y)^{2s} = \binom
{2s}{s}x^sy^{s}.
\end{equation}
\end{corollary}
\begin{proof}
Set $r=s$, $d=2s$ and $m = s+1$ in \eqref{E:basicsynch}. Since $|\frac rm| = |\frac{d-r}
m|  < 1$, the summation on the right-hand side has a single term, $i=0$, and \eqref{E:basicsynch}
becomes
\[
\frac1{s+1}\cdot \sum_{j=0}^{s} \ze_{s+1}^{-sj} ( x + \ze_{s+1}^jy)^{2s} = \binom{2s}s x^sy^s;
\]
\eqref{E:monom8} follows from 
$ \ze_{s+1}^{-sj} ( x + \ze_{s+1}^jy)^{2s}=  \ze_{2s+2}^{-2sj}( x +\ze_{2s+2}^{2j}y)^{2s} =
(\ze_{2s+2}^{-j} x + \ze_{2s+2}^{j}y)^{2s}$.
\end{proof}

 \begin{proof}[Proof of Theorem \ref{MT}(11) for even $d$]
Take $(x,y) \mapsto (x^2,y^2)$ in  \eqref{E:monom8}, to obtain
\begin{equation}\label{E:monom}
\sum_{j=0}^{s} (\ze_{2s+2}^{-j} x^2 + \ze_{2s+2}^{j}y^2)^{2s} = (s+1) \binom {2s}{s}(xy)^{2s},
\end{equation}
 a linear dependence among $s+2$ $2s$-th powers of an honest set of
quadratic forms. 
\end{proof}

If $s=2v$, we have $(\ze_{4v+2}^{-j},\ze_{4v+2}^{-j}) = 
((-\zeta_{2v+1}^v)^j,(-\zeta_{2v+1}^{v+1})^j)$, so   
\begin{equation}\label{E:monom2}
\sum_{j=0}^{2v} ((\zeta_{2v+1}^v)^j x^2 +(\zeta_{2v+1}^{v+1})^j y^2)^{4v} = (2v+1)\binom
{4v}{2v}(xy)^{4v}.
\end{equation}
When $s=1$, we have $\ze_2 = -1$ and \eqref{E:monom2} reduces to \eqref{E:pyt};
when $s=2$ and 3,  \eqref{E:monom2} becomes \eqref{E:quartic} and \eqref{E:six}.
Taking $(x,y) \mapsto (e^{-i\theta}(x+iy),e^{i\theta}(x-iy))$ in \eqref{E:monom8}  
(see \cite[(5.8)]{Re2}, which is incorrect -- unfortunately missing the factor of $2^{-2s}$) gives
\begin{equation}\label{E:kent2}
\frac 1{s+1}\sum_{j=0}^s
\left(\cos(\tfrac{j\pi}{s+1}+ \theta) x +  \sin(\tfrac{j\pi}{s+1}+
  \theta) y\right)^{2s} = \frac 1{2^{2s}}\binom {2s}s (x^2+y^2)^s, \quad \theta \in \cc.
\end{equation}
With $\theta \in \mathbb R$, \eqref{E:kent2} was a 19th century quadrature formula; 
see the discussion after \cite[Cor.5.6]{Re2} for details. 
Taking $\theta \in \mathbb R$  and $(x,y) \mapsto (x^2 - y^2, 2 x y)$, so that 
$x^2+y^2 \mapsto (x^2 + y^2)^2$ in \eqref{E:kent2}, gives a nice family of $\mathcal 
W_2(s+2,2s)$ cousins in $\mathbb R[x,y]$.

There doesn't seem to be such a simple proof of Theorem \ref{MT}(11) for odd $d$, and we need 
to introduce powers of trinomials as summands.  More generally, it is useful to 
present two quadratic cases, which are 
corollaries of Theorem \ref{T:synching}; note that
\[
\zeta_m^{-rj}(\zeta_m^{-j} x^2 + \al x y + \zeta_m^{j} y^2)^d = \zeta_m^{-(r+d)j}
(x^2 + \al \zeta_m^{j}x y + \zeta_m^{2j} y^2)^d,
\]
gives \eqref{E:quadsynch} the shape of  Theorem \ref{T:synching}  
for $p(x,y) = (x^2 + \al x y + y^2)^d$.

\begin{corollary}\label{T:quadsynching}
Suppose $d, m \in \nn, v \in \zz$ and $\al \in \mathbb C$. Let
\begin{equation}\label{E:quadsynch}
\Psi(v,m,d;\al):= \frac 1m \sum_{j=0}^{m-1} \zeta_m^{-vj}
(\zeta_m^{-j} x^2 + \al x y + \zeta_m^{j} y^2)^d. 
\end{equation}
(i) If $m > d$, then 
\begin{equation}\label{E:quadsynch1}
\Psi(0,m,d;\al) = 
\left( \sum_{r=0}^{\lfloor d/2 \rfloor} \frac{d!}{(r!)^2(d-2r)!}\al^{d-2r}\right) x^dy^d.
\end{equation}

(ii) If $2m > d \ge m$, then 
\begin{equation}\label{E:quadsynch2}
\begin{gathered}
\Psi(0,m,d;\al) = \left( \sum_{r=0}^{\lfloor d/2 \rfloor} \frac{d!}{(r!)^2(d-2r)!}\al^{d-2r}\right)  x^dy^d \\
+ \left( \sum_{r=0}^{\lfloor (d-m)/2 \rfloor} \frac{d!}{r!(r+m)!(d-m-2r)!}\al^{d-m-2r}  \right) 
(x^{d+m}y^{d-m} + x^{d-m}y^{d+m}). 
\end{gathered} 
\end{equation}
\end{corollary}

\begin{proof}

By the trinomial theorem, 
\[
(x^2 + \al x y + y^2)^d = \sum_{r+s+t = d} \frac{d!}{r!s!t!} \al^sx^{2r+s}y^{s+2t};
\]
note that $(2r+s,s+2t) = (2d-i, i) \iff r - t = d-i$; all sums can only be taken over $r,s,t \ge 0$.
In each case, $m$ is relatively large compared to $d$ and very few terms will be nonzero. In
(i), $x^{2d-i}y^i$ appears when $i \equiv d \mod m$. 
Since $d < m$, this only occurs when $i=d$, so $r=t$ and the 
coefficient of $x^dy^d$ is found by summing $\frac{d!}{r!s!t!}\al^s$ over
the set  $\{(r,s,t) = (r, d-2r, r)\}$. Similarly, in (ii), $v=0$ and $2m > d$, so we have 
three cases: $r-t \in\{ -m, 0, m\}$, and the terms sum as indicated. 
\end{proof}

We use \eqref{E:quadsynch2}  when $d-m \ge 2$ by choosing $\al=\al_0$ to be a non-zero
root of the polynomial coefficient of $(x^{d+m}y^{d-m} + x^{d-m}y^{d+m})$,
so that the terms on both sides of the expression are $d$-th 
powers. 
In general, the Klein set of $\Psi(v,m,d;\al)$ will consist of two parallel regular $m$-gons, whose
altitude and relative orientation depends on $\al$. If $(xy)^d$ appears in the identity, then the
two poles are added. 

 \begin{proof}[Proof of Theorem \ref{MT}(11) for odd $d$]
Suppose $d=2s+1 \ge 5$. We have
\begin{equation}\label{E:Main(9)}
\begin{gathered}
\Psi(0,s+1,2s+1;\al) = 
\sum_{j=0}^{s} (\zeta_{s+1}^{-j} x^2 + \al x y + \zeta_{s+1}^{j} y^2)^{2s+1}  =\\
A_s(\al)x^{3s+2}y^s  + B_s(\al)x^{2s+1}y^{2s+1}+  A_s(\al)x^sy^{3s+2} ; \\
A_s(\al) =  \binom{2s+1}{s} \al^s + (2s+1)\binom{2s}{s-2}\al^{s-2} + \dots.
\end{gathered}
\end{equation}
Let $\al = \al_0$ be a non-zero root of $A_s(\al)$; this exists because $s \ge 2$, so
 \eqref{E:Main(9)} becomes
\[
\Psi(0,s+1,2s+1;\al_0) = B(\al_0)(xy)^{2s+1},
\]
which is a sum of $s+1$ $(2s+1)$-st powers equal to another $(2s+1)$-st power. 
 \end{proof}
\begin{proof}[Alternate Proof of Theorem \ref{MT}(11) for $d = 2s, s \ge 3$]
Suppose $s \ge 3$. Then
\begin{equation}\label{E:Main(92)}
\begin{gathered}
\Psi(0,s+1,2s;\al) =
\tilde A_s(\al)(x^{3s+1}y^{s-1}+ x^{s-1}y^{3s+1}) + \tilde B_s(\al)x^{2s}y^{2s}; \\
\tilde A_s(\al) =  \binom{2s}{s-1} \al^{s-1} + (2s)\binom{2s-1}{s-3}\al^{s-3} + \dots
\end{gathered}
\end{equation}
Again, choose $\al= \al_0$ to be a non-zero root of $\tilde A_s$.
\end{proof}
By looking at the pattern of linear dependence among the elements, 
it is not hard to show that the families in  \eqref{E:monom} and \eqref{E:Main(92)} are not cousins.

Here are other synching examples; \eqref{E:quadsynch1} 
 requires $m > d$. We have $\Psi(0,4,3;\al) = (\al^3+6\al)x^3y^3$, 
so $\Psi(0,4,3,\sqrt{-6})$ gives a $\mathcal W_2(4,3)$ set. 
In (ii) we need $d \in [m+2,2m)$. For $m=3$, this implies that  $d=5$, and we 
obtain a variant of \cite[(4.12)]{Re1}:
\begin{equation}\label{E:quinticsyn1}
\begin{gathered}
3\Psi(0,3,5;\al)  = \sum_{j=0}^2 (\om^k x^2 + \al x y + \om^{-k}y^2)^5\\  =
15(1+2\al^2)(x^8y^2 + x^2y^8) + 3\al(\al^4 + 20\al^2+30) x^5y^5 \implies\\
\Psi\left(0,3,5;\sqrt{-1/2}\right) = \bigl(\sqrt{-9/2}\ xy\bigr)^5.
\end{gathered}
\end{equation}
The linear change 
 $(x,y) \mapsto (\sqrt{-2}\ x -(1 + \sqrt{3})y,  -(1 + \sqrt{3})x + \sqrt{-2}\ y)$,
applied to \eqref{E:quinticsyn1}, gives $3(1+\sqrt{3})$ times a flip of \eqref{E:quintic11}.
The Klein set here is again a cube, rotated so the vertices are the two poles and antipodal
equilateral triangles at $z = \pm\frac 13$.

For $m=4$, the possibilities are $d=6,7$; we have
\[
4\Psi(0,4,6; \sqrt{-2/5})=\sum_{k=0}^3 (i^{-k}  x^2 +\sqrt{-2/5}\ xy +i^{k}y^2)^6 =
11\cdot(\sqrt{-8/5}\ x y)^6; 
\]
$\Psi(0,4,7; \sqrt{-6/5})$ is just \eqref{E:seven}.

Two other examples show the range of Corollary
 \ref{T:quadsynching}. First,
\[
\begin{gathered}
4\Psi(2,4,4;\al) = \sum_{j=0}^3 (-1)^k(i^{-k} x^2 + \al x y + i^{k}y^2)^4 = 8(2 + 3\al^2) 
(x^6 y^2 + x^2 y^6).\end{gathered}
\]
On taking $\al = \al_0 = \sqrt{-2/3}$, transposing two terms to get two equal sums of two 
fourth powers, and after multiplying  through by $\sqrt{3}$, we obtain \eqref{E:quartic12}.
For $d=5$, we may recover \eqref{E:quintic11} as $4\Psi(2,4,5,\sqrt{-2})$ from 
\[
4\Psi(2,4,5;\al) = \sum_{j=0}^3 (-1)^k (i^{-k} x^2 + \al x y + i^{k}y^2)^5 =
40\al(2+\al^2)(x^7 y^3 + x^3 y^7). 
\]

An unusual phenomenon occurs with $\Psi(0,5,14; \al)$: by the general method,
\[
\Psi(0,5,14; \al) = A(\al)(x^{24}y^4 + x^4y^{24}) + B(\al)(x^{19}y^9 + x^9y^{19}) + C(\al)x^{14}y^{14}. 
\]
It turns out that $A(\al)$ and $B(\al)$ have a common factor $1+\al^2$. Upon setting $\al = i$,
we obtain \eqref{E:14}. A computer search has not found other examples of this phenomenon.
As noted earlier, the Klein form of \eqref{E:14} is an icosahedron, but  
an icosahedron
can be rotated so that its vertices lie in four horizontal equilateral triangles. This suggests that 
\eqref{E:quintic11} should be the cousin of a union of two $\Psi(v,3,14;\al)$'s. Indeed,
with $\phi = \frac{1+\sqrt 5}2$ as usual,
\begin{equation}\label{E:142}
\sum_{k=0}^2 (\om ^k x^2 + \phi^2 x y - \om^{-k}y^2)^{14} + 
\sum_{k=0}^2 (\om ^k \phi  x^2 - \phi^{-1} x y -
 \om^{-k}\phi y^2)^{14} = 0.
\end{equation} 
The Sch\"onemann coefficients of the icosahedron, 
$\{(\phi^2+1)^{-1/2}\cdot(\pm \phi, \pm 1, 0)\}$  and their cyclic 
images, lead to yet another cousin of \eqref{E:quintic11}: 
\begin{equation}\label{E:143}
\begin{gathered}
(x^2 + 2\phi x y - y^2)^{14} + (x^2 - 2\phi  x y - y^2)^{14}  
+ ( (\phi + i)(x^2 - \tfrac{1-2i}{\sqrt 5} y^2))^{14} \\ + ( (\phi- i)(x^2 - \tfrac{1+2i}{\sqrt 5} y^2))^{14} 
= (\phi x^2 + 2 i x y + \phi y^2)^{14} + (\phi x^2 - 2 i x y + \phi y^2)^{14}.
\end{gathered}
\end{equation}
The corresponding quadratics for a dodecahedron, alas,  give a $\mathcal W_2(10,14)$ set.

There is no reason for synching to be limited to trinomials. Here is an example of a
$\mathcal W_4(4,3)$ family of linearly independent elements:
\begin{equation}\label{E:quarcube}
\sum_{k=0}^3 (-1)^k(x^4 + i^k\sqrt{6}\ x^3y -6i^{2k}x^2y^2 -\sqrt{6} \ i^{3k} x y^3 + y^4)^3 = 0;
\end{equation}
the quartics are linearly independent.

Finally, we compare Theorem \ref{MT}(5), (6) and (11). The bound in (11) is linear in $d$ and 
weaker than (5). This leads to  the natural question: what is the smallest $d$ so that $k \ge 2$ 
and $\Phi_{k+1}(d) < \Phi_{k}(d)$?   Taking Theorem \ref{MT}(7), (10) and (11), into account, we 
must have $d \ge 6$, and the smallest $d$  for which (5) or (6) beats the bound for $k=2$ in (11) is 
$d=15$:  $1 + \lfloor \sqrt{61} \rfloor = 8 < 9 = 2 + \lfloor   15/2\rfloor$.

\section{Overview of $\mathcal W_2(4,d)$ sets and tools.}

In order to prove Theorem \ref{MT}(8), we need an abbreviated version of Sylvester's algorithmic 
theorem from 1851 on the representation of forms as a sum of powers of linear forms. We refer the 
reader to \cite[Thm.2.1]{Re2} for the general theorem and proof.

\begin{theorem}[After Sylvester]\label{T:syl}
Suppose $d \ge 3$ and
\begin{equation}\label{E:squaring}
p(x,y) = \sum_{j=0}^d \binom dj a_j x^{2d-2j}y^{2j}, \qquad 
q(x,y) = \sum_{j=0}^d \binom dj a_j x^{d-j}y^{j}.
\end{equation}
Then $p$ is a sum of $d$-th powers of two honest even quadratic forms 
if and only if  there 
exists a non-square quadratic form
$h(u,v) = c_0u^2 + c_1uv + c_2v^2 \neq 0$ so that
\begin{equation}\label{E:rank2}
\begin{pmatrix}
a_0 & a_1 & a_2 \\
a_1 & a_2 & a_3\\
\vdots & \vdots  & \vdots \\
a_{d-2}& a_{d-1} & a_d
\end{pmatrix}
\cdot
\begin{pmatrix}
c_0\\c_1\\ c_2
\end{pmatrix}
=\begin{pmatrix}
0\\0 \\\vdots \\ 0
\end{pmatrix}.
\end{equation}
\end{theorem}
\begin{proof}[Sketch of Proof]
A comparison of the coefficients of monomials in $p$ and $q$ shows that 
\[
\begin{gathered}
p(x,y) =  (\al_1 x^2 + \be_1 y^2)^d +  (\al_2 x^2 + \be_2 y^2)^d    \iff \\
q(x,y) =  (\al_1 x + \be_1 y)^d +  (\al_2 x + \be_2 y)^d.
\end{gathered}
\]
Assuming $\al_j \neq 0$, $q(x,y) = (\al_1 x + \be_1 y)^d +  (\al_2 x + \be_2 y)^d$
implies that $a_j = \la_1 \ga_1^j + \la_2 \ga_2^j$, 
where $\la_i = \al_i^d$ and $\ga_i = \be_i/\al_i$, so $(a_j)$ satisfies the 
linear recurrence given by
\eqref{E:rank2} with $c_0 = \ga_1\ga_2, c_1 = -(\ga_1+\ga_2), c_2 = 1$; $h(u,v) =
(\ga_1u - v)(\ga_2u - v)$. Conversely, any solution $(a_j)$ to this recurrence has the
 indicated shape. If $\al_2=0$, then $\al_1 \neq 0$ by honesty; $a_j = \la_1 \ga_1^j $ for $j \le d-1$ 
and \eqref{E:rank2} holds with $h(u,v) = u(\ga_1 u - v) $.
\end{proof}
The matrix in  \eqref{E:rank2} is called the {\it 2-Sylvester matrix} for $p$ (or $q$).   
A necessary condition for $p$ to be a sum of two $d$-th powers is that the 2-Sylvester matrix of 
$p$ (with $d-2$ rows) has  rank $\le 2$. As $d$ increases, this becomes 
increasingly harder.

We also need a special case of a classical result about simultaneous
 diagonalization; there doesn't seem to be an easy-to-find modern proof.  
 
 \begin{theorem}[Diagonalization]\label{T:diag}
If $f_1$ and $f_2$ are relatively prime binary quadratic forms,
then there is a linear change $M$ so that $f_1\circ M$ and $f_2\circ M$ are both even. 
\end{theorem}
\begin{proof}
Suppose without loss of generality that $rank(f_1) \ge
rank(f_2)$. If $rank(f_1) = 1$, then $(f_1,f_2) = (\ell_1^2,\ell_2^2)$ and a linear change
takes $(\ell_1,\ell_2) \mapsto (x,y)$. Otherwise, there exists $M_1$ so that
$(f_1\circ M_1)(x,y) = x^2 + y^2$ and 
$(f_2\circ M_1)(x,y) = ax^2 + bxy + cy^2$. Since these are relatively prime, $a  \pm ib - c \neq 0$. 

Drop ``$M_1$", and observe that for any $z \in \mathbb C$, $f_1$ is fixed by any 
orthogonal  linear change $M_z: (x,y) \mapsto   ((\cos z) x + (\sin z) y, - (\sin z) x +( \cos z) y)$,
under which the coefficient of $xy$ in $f_2\circ M_z$ is
$  (a - c)\sin 2z + b\cos 2z$. If $a=c$, let $z = \frac{\pi}4$. Otherwise, choose $z$ so 
that $\tan 2z = -\frac b{a-c}$; this is possible, since the range of $tan(z)$ is $\cc \setminus\{\pm i\}$. 
The coefficient of $xy$ in $f_2\circ M_z$ vanishes, so $f_1\circ M_z, f_2\circ M_z$ are
both even. \end{proof}

Suppose $d \ge 3$ and we have a $\mathcal W_2(4,d)$ set, flipped and normalized so that
\begin{equation}\label{E:main}
p(x,y) = f_1^d(x,y) + f_2^d(x,y) =  f_3^d(x,y) + f_4^d(x,y),
\end{equation}
for an honest set $\{f_1,f_2,f_3,f_4\}$ of binary quadratic forms.

\begin{theorem}\label{T:even}
If \eqref{E:main} holds, then there exists a linear change after which both $f_1$ and $f_2$ 
are even, so $p$ is even. We have $\gcd(f_1,f_2) = \gcd(f_3,f_4) = 1$, but it is not true
 that $f_3,f_4$ are both even.
\end{theorem}
\begin{proof}
If $\gcd(f_1,f_2) =\ell$ for a linear form $\ell$, so that $f_1 = \ell \ell_1$ and 
$f_2 = \ell \ell_2$, then
\[
\ell^d \ | \  f_3^d + f_4^d = \prod_{k=0}^{d-1} (f_3 + \zeta_d^k f_4).
\]
Since $ d \ge 3$, $\ell$ must divide at least two different quadratic factors on the right, 
say $\ell \ | \ f_3 + \zeta_d^{k_1} f_4, f_3 + \zeta_d^{k_2} f_4$ for $k_1 \neq k_2$. This 
implies that $\ell \ | \ f_3, f_4$ and $f_3= \ell \ell_3$ and $f_4 = \ell\ell_4$ for linear 
$\ell_3, \ell_4$. hence we can factor $\ell^d$ from \eqref{E:main} to obtain
$\ell_1^d + \ell_2^d = \ell_3^d + \ell_4^d$,
which contradicts Theorem \ref{T1}, since $d \ge 3$. Similarly, $\gcd(f_3,f_4) = 1$.

Thus $f_1$ and $f_2$ are relatively prime, and by Theorem \ref{T:diag}, we may 
simultaneously diagonalize them, after which (dropping $M$),
\[
p(x,y) = (\al_1 x^2 + \be_1 y^2)^d + (\al_2 x^2 + \be_2 y^2)^d = f_3^d(x,y) + f_4^d(x,y).
\]
Suppose $f_3(x,y) = \al_3 x^2 + \be_3 y^2$ and $f_4(x,y) = \al_4 x^2 + \be_4 y^2$ are 
both even. Then
\begin{equation}\label{E:shrink}
\begin{gathered}
(\al_1 x^2 + \be_1 y^2)^d + (\al_2 x^2 + \be_2 y^2)^d = (\al_3 x^2 + \be_3 y^2)^d 
+ (\al_4 x^2 + \be_4 y^2)^d \\
\implies (\al_1 x + \be_1 y)^d + (\al_2 x + \be_2 y)^d = (\al_3 x + \be_3 y)^d 
+ (\al_4 x + \be_4 y)^d. 
\end{gathered}
\end{equation}
Since $\{f_j\}$ is honest, \eqref{E:shrink}
violates Theorem \ref{T1},  so $f_3$ and $f_4$ are not both even.
\end{proof}

Here then is our strategy. We seek to find all pairs $\{f_3,f_4\}$ which are not both even 
but for which $f_3^d + f_4^d$ {\it is} even. Then, from among those, we need to find 
those which can {\it also} be written as a sum of two $d$-th powers of even quadratic forms.

How can it happen that $f_3^d + f_4^d$ is even when at least one of $\{f_3,f_4\}$ is not even? 
Two cases come readily to mind:
\begin{equation}\label{E:tame}
(a x^2 + b x y + c y^2)^d + (a x^2 - b x y + c y^2)^d,
\end{equation}
and, if $d$ is even, 
\begin{equation}\label{E:tame2}
(a x^2 + c y^2)^d + b(xy)^d.
\end{equation}
We call \eqref{E:tame} and \eqref{E:tame2} the {\it tame} cases; otherwise $\{f_3,f_4\}$ are in the
{\it wild} case. 
There is an important  practical distinction. The tame expressions are formally symmetric
under $y \mapsto -y$, but wild expressions are not. Thus, any wild \eqref{E:main}
implies the 
existence  of a  {\it third}  representation for $p$ a sum of two $d$-th powers.

The case $d=3$ is best handled by other techniques and is covered in the companion
paper \cite{Re3}. In preparation for implementing this strategy, we calculate the tame and
wild cases which might occur from the list   of
$\mathcal W_2(4,d)$ sets for $d \ge 4$ in Theorems \ref{W4} and \ref{W5}.
 Each identity \eqref{E:main} has two flips: $f_1^d - f_3^d =  f_4^d-
f_2^d$ and $f_1^d - f_4^d =  f_3^d-f_2^d$, and since either side can be diagonalized, there are 
potentially six cases. (If there are three equal sums, there are potentially fifteen cases.)
Fortunately, symmetry reduces the number of cases substantially. 

\begin{theorem}\label{T:44}
\

(i) The diagonalizations of \eqref{E:quartic} are, up to scaling, 
\begin{equation}\label{E:tame43}
\begin{gathered}
(x^2 + y^2)^4 -18(xy)^4 = -(\om x^2 + \om^2 y^2)^4 - (\om^2 x^2 + \om y^2)^4 \\
= x^8 + 4x^6y^2 - 12x^2y^2 + 4x^2y^6 + y^8,
\end{gathered}
\end{equation}
and
\begin{equation}\label{E:tame41}
\begin{gathered}
- (2x^2 + 2y^2)^4 + 18(x^2 - y^2)^4  \\
= (x^2 + 2\sqrt{-3}\ x y + y^2)^4 + (x^2 -2\sqrt{-3}\ x y + y^2)^4 \\
= 2(x^8 - 68x^6y^2+6x^4y^4-68x^2y^6+y^8).
\end{gathered}
\end{equation}
(ii) The diagonalizations of \eqref{E:quartic13} are, up to scaling, 
\begin{equation}\label{E:late4}
\begin{gathered}
(\al x^2  - \be y^2)^4 - (\be x^2 - \al y^2)^4 \\ = (\om x^2 - \sqrt 3\ xy - \om^2y^2)^4 - 
(\om^2 x^2 - \sqrt 3\ xy - \om y^2)^4 \\ 
= (\om x^2 + \sqrt 3\ xy - \om^2y^2)^4 - 
(\om^2 x^2 +\sqrt 3\ xy - \om y^2)^4 \\
= \sqrt{-3}\left( x^8 - 14x^6y^2+14x^2y^6 - y^8\right),
\end{gathered}
\end{equation}
where $\al = \frac {2+\sqrt{-3}}2, \be =  \tfrac {2-\sqrt{-3}}2$; and
\begin{equation}\label{E:4last2}
\begin{gathered}
((1 + \sqrt{-6})x^2 + (1 - \sqrt{-6})y^2)^4 + ((1 - \sqrt{-6})x^2 + (1 + \sqrt{-6})y^2)^4 \\
= (x^2 + 2\sqrt{-6}\ xy + y^2)^4 + (x^2 - 2\sqrt{-6}\ xy + y^2)^4\\ 
= 2(x^8 - 140x^6y^2 + 294x^4y^4 - 140x^2y^6 + y^8).
\end{gathered}
\end{equation}

(iii) The diagonalization of \eqref{E:quintic11} is, up to scaling,
\begin{equation}\label{E:tame51}
\begin{gathered}
((1 -  \sqrt{-2}) x^2 + (1 + \sqrt{-2}) y^2)^5  + ((1 + \sqrt{-2}) x^2 + (1 - \sqrt{-2}) y^2)^5 \\
=  (x^2 - 2\sqrt{-2}\  x y + y^2)^5 + (x^2 + 2\sqrt{-2}\  x y + y^2)^5 = \\
=2(x^{10} - 75x^8y^2+90x^6y^4+90x^4y^6-75x^2y^8+y^{10}).
\end{gathered}
\end{equation}
\end{theorem}
\begin{proof}
(i) First, in \eqref{E:quartic}, the summands on the left are cyclically permuted by 
$(x,y) \mapsto (\om x, \om^2 y)$, so there is only one choice up to scaling. 
One is already diagonalized as in \eqref{E:tame43}. To diagonalize the left-hand side
in \eqref{E:tame43}, take $(x,y) \mapsto (x+y, x-y)$ and multiply through
by $-1$, to obtain \eqref{E:tame41}.

(ii) It is convenient to name the forms from  \eqref{E:quartic13} in \eqref{E:413} .  Let 
\begin{equation}\label{E:413}
\begin{gathered}
f_{1,1}(x,y) = x^2 + \sqrt 3\ xy - y^2, \quad
f_{1,2}(x,y) = x^2 - \sqrt 3\ xy - y^2,  \\
f_{1,3}({x,y})  = f_{1,1}(\om^2 x,\om y), \quad 
f_{1,4}({x,y}) = f_{1,2}(\om^2 x,\om y),\\
 f_{1,5}({x,y})  = f_{1,1}(\om x,\om^2 y), \quad f_{1,6}({x,y})  = 
 f_{1,2}(\om x,\om^2 y); \\
  f_{1,1}^4 - f_{1,2}^4 = f_{1,3}^4 - 
f_{1,4}^4 = f_{1,5}^4 - f_{1,6}^4 = 8 \sqrt 3\  xy\ (x^6 - y^6).
\end{gathered}
\end{equation}
Let $M_1$ denote the linear change $(x,y) \mapsto (\om^2 x, \om y)$, so that 
$M_1$ cycles $f_{1,1}\mapsto f_{1,3} \mapsto f_{1,5} \mapsto f_{1,1}$ and
$f_{1,2}\mapsto f_{1,4} \mapsto f_{1,6} \mapsto f_{1,2}$. 
Let $M_2$ denote the linear change  $(x,y) \mapsto \frac 1{\sqrt{2}}(x + iy, ix + y)$, which
has two nice properties. First, $M_2$ cycles $f_{1,3}\mapsto f_{1,5}\mapsto f_{1,6} \mapsto
f_{1,4} \mapsto f_{1,3}$, but it also takes
$(f_{1,1},f_{1,2}) \mapsto (\al x^2  - \be y^2, \be x^2 - \al y^2)$.
On the Riemann sphere, $M_1$ induces a $\frac{2\pi}3$ rotation on the axis of the poles.
and $M_2$ induces the rotation taking $(a,b,c) \mapsto (a,c,-b)$.

By repeatedly using $M_1$ and $M_2$, the fifteen pairs
$\{f_{1,i},f_{1,j}\}$ which might be simultaneously diagonalized given the identity
$f_{1,3}^4-f_{1,4}^4 = f_{1,5}^4-f_{1,6}^4$, reduce to two cases, after 
linear changes. We have already seen one: $M_2$ diagonalizes
\eqref{E:quartic13} into \eqref{E:late4}. 

For the other, note that 
\begin{equation}\label{E:4last1}
\begin{gathered}
f_{1,4}^4(x,y) + f_{1,5}^4(x,y) = f_{1,3}^4(x,y) + f_{1,6}^4(x,y)\\ = 
-(x^8 + 14 x^6y^2 + 42 x^4y^4+14x^2y^6+y^8).
\end{gathered}
\end{equation}
An appeal to Theorem \ref{T:syl} shows that the octic in \eqref{E:4last1} is {\it not} a sum of two
fourth powers of even quadratic forms. Under the linear change $M_3$, 
which takes $(x,y) \mapsto ( x -(\sqrt{2} - 1)y,   i(\sqrt{2} - 1)x+ iy)$ and division by 
$\sqrt 2 - 2$, \eqref{E:4last1} becomes \eqref{E:4last2}.

(iii) We name the quadratics from  \eqref{E:quintic11} in \eqref{E:415}. Let
$M_4$ be the scaling $(x,y) \mapsto (\zeta_8 x, \zeta_8^3y)$, which takes $(x^2,xy,y^2)
\mapsto (ix^2, -xy, -iy^2)$, so that
\begin{equation}\label{E:415}
\begin{gathered}
f_{2,1}(x,y) = x^2 + \sqrt{-2} \ x y + y^2, \ f_{2,2} = f_{2,1}\circ M_4, \
 f_{2,3} = f_{2,2} \circ M_4, \\ f_{2,4} = f_{2,3} \circ M_4; 
  \qquad f_{2,1}^5 + f_{2,2}^5 + f_{2,3}^5 + f_{2,4}^5 = 0.
\end{gathered}
\end{equation}
Thus $M_4$ cycles $f_{2,1} \mapsto f_{2,2} \mapsto f_{2,3} \mapsto f_{2,4} \mapsto f_{2,1}$.
The symmetry of the Klein set for $\{f_{2,j}\}$ (the cube) suggests that we
let $M_5$ be the 
linear change $(x,y) \mapsto \tfrac 1{\sqrt{2}} \cdot(-x+\zeta_8^5y,\zeta_8^3x + y)$. Then
$ M_5$ fixes $f_{2,1}$ and $f_{2,4}$ and permutes $f_{2,2}$ and $f_{2,3}$.

Thus $M_4$ maps the flip $f_{2,1}^5 + f_{2,2}^5 = -f_{2,3}^5-f_{2,4}^5$ into 
$f_{2,2}^5 + f_{2,3}^5 = -f_{2,4}^5-f_{2,1}^5$ and $ M_5$ maps it into 
$f_{2,1}^5 + f_{2,3}^5 = -f_{2,2}^5-f_{2,4}^5$, so, up to cousin, we need only
consider one flip. The easiest one to deal with is $f_{2,1}^5 + f_{2,3}^5 = -f_{2,2}^5-f_{2,4}^5$. 
This is
\begin{equation}\label{E:flippedquintic}
\begin{gathered}
(x^2 + \sqrt{-2}\ x y + y^2)^5 + (-x^2 + \sqrt{-2}\  x y - y^2)^5 \\ = -(i x^2 - \sqrt{-2}\ x y -i y^2)^5 - 
(-i x^2 - \sqrt{-2}\ x y + i y^2)^5\\  = 2\sqrt{-2}\ xy(5x^8 - 6x^4y^4+5y^8).
\end{gathered}
\end{equation}
Upon taking $(x,y) \mapsto (x+iy,x-iy)$, and dividing by $\sqrt{-2}$, 
\eqref{E:flippedquintic} becomes \eqref{E:tame51}. And under the linear change,
$(x,y) \mapsto \frac1{\sqrt{2}}(x + i y, x - i y)$, \eqref{E:quartic12} also becomes  \eqref{E:tame51}.
The Klein set of the summands in \eqref{E:tame51}
is a rotated cube lying in the planes $y = \pm \sqrt{1/3}$, 
so that the edge $(0,\pm\sqrt{1/3},\sqrt{2/3})$  lies on top. 
\end{proof}

\section{Finishing the proof}

We first make a simplifying observation in the tame case. If $(f_3,f_4)$ is given in \eqref{E:tame} or 
\eqref{E:tame2}  and $a=0$ (or $c=0$), then $f_3$ and $f_4$ have a common factor of 
$y$ (or $x$), violating Theorem \ref{T:even}. Similarly, we may assume that $b \neq 0$. 
Thus, after scaling, we may assume that \eqref{E:tame} and \eqref{E:tame2}  take the shape
\begin{equation}\label{E:tamer1}
(x^2 + b x y + y^2)^d + (x^2 - b x y + y^2)^d, \qquad b \neq 0;
\end{equation}
\begin{equation}\label{E:tamer2}
(x^2 + y^2)^{2e} + b\binom {2e}{e}(xy)^{2e}, \qquad b \neq 0.
\end{equation}
\begin{theorem}\label{T:tameanswer}
The only $\mathcal W_2(4,d)$ sets which come from a tame representation for $d \ge 4$ are
given in Theorem \ref{T:44} by \eqref{E:tame43}, \eqref{E:tame41}, \eqref{E:4last2}, and 
 \eqref{E:tame51}. These
sets are all cousins or sub-cousins of the families in Theorems \ref{W4}, \ref{W5}. 
\end{theorem}
\begin{proof}
We analyze \eqref{E:tamer2} first. The 2-Sylvester matrix of $(x^2 + y^2)^4 + 6b(xy)^4$ is
\begin{equation}\label{E:tame2syl}
\begin{pmatrix}
1 & 1 & 1+b\\
1 & 1+b &1\\
1+b &1 & 1
 \end{pmatrix},
\end{equation}
 which has rank 2 only if $-b^2(b+3)= 0$; if $b=-3$, we obtain \eqref{E:tame43}.

If $d = 2s \ge 6$ and $p_{2s,b}(x,y) = (x^2 + y^2)^{2s} + b \binom {2s}s (xy)^{2s}$, 
then the $(2s-1) \times 3$ 2-Sylvester
matrix  consists of \eqref{E:tame2syl}, with $s-2$ rows of $(1,1,1)$ appended both at the 
top and the bottom. Such a matrix has rank 2 only if $b=0$. 

For \eqref{E:tamer1}, we first observe that
\begin{equation}\label{E:tamer1cheat}
(x^2 + b x y + y^2)^d + (x^2 - b x y + y^2)^d =
2\sum_{0 \le i \le d/2} \binom d{2i}(x^2 + y^2)^{d-2i}(xy)^{2i}.
\end{equation}
Suppose $d=4$. Then the sum in \eqref{E:tamer1cheat} becomes
\[
\begin{gathered}
2x^8 + (8+12b^2)x^6y^2 + (12 + 24b^2 + 2b^4)x^4y^4 +(8+12 b^2) x^2 y^6 + 2 y^8.
\end{gathered}
\]
Apply Theorem \ref{T:syl}: the
2-Sylvester matrix has discriminant $-\frac{b^8}{27}(12 + b^2) (24 + b^2)$, and has rank 2
only if  $b^2 \in \{-12,-24\}$. These cases are presented in \eqref{E:tame41} and \eqref{E:4last2},
 and are a cousin of \eqref{E:quartic} and a sub-cousin of \eqref{E:quartic13}, respectively.

Suppose $d=5$. Then applying Theorem  \ref{T:syl}  to  \eqref{E:tamer1cheat} gives a
$4 \times 3$ matrix; computing the $3\times 3$ minors shows that the matrix has rank 2 only when 
$b=0$ or $b^2 = -8$. Taking $b = \sqrt{-8}$, we obtain \eqref{E:tame51}, which is a cousin of 
\eqref{E:quintic11}.

Now suppose $d \ge 6$; \eqref{E:tamer1cheat}  gives
\[
\begin{gathered}
a_0= a_d = 2, \quad a_1 = a_{d-1} = 2 + b^2(d-1), \\ 
a_2 = a_{d-2} = 2 + b^2(d-2)(12+(d-3)b^2)/6, \\
a_3 = a_{d-3} = 2 + b^2(d-3)(180 + b^2(30 d - 120) +b^4(d^2-9d+20))/60.
\end{gathered}
\]
The submatrix of the 2-Sylvester matrix consisting of the first and last two rows is
\[
\begin{pmatrix}
a_0&a_1&a_2 \\
a_1 &a_2 & a_3 \\
a_3 & a_2 & a_1 \\
a_2 & a_1 & a_0
\end{pmatrix}.
\]
The 1,2,4 minor of this sub-matrix is
$- \frac{b^8}{9(d-1)}\binom{d+1}5 (12 + b^2(d-3))(24 + b^2(2d-7)).$
If $b^2 = - \frac{12}{d-3}$, then the 1,2,3 minor becomes 
$\frac{55296\ d^2(d+1)(d-4)}{25(d-3)^5} \neq 0.$
However, if $b^2 = - \frac{24}{2d-7}$, then all four minors vanish. (Note that 
$d=4,5$ then give $b^2=-24, b^2 = -8$, which we have already seen.) 
We re-compute the $a_k$'s for  $b^2 = - \frac{24}{2d-7}$, and find that
the first three rows of the 2-Sylvester matrix give
\[
\begin{vmatrix}
a_0&a_1&a_2 \\
a_1 &a_2 & a_3 \\
a_2 & a_3 & a_4
\end{vmatrix} = - \frac{3538944 (d-5) (d-4) d (1 + d) (2d-1)^2}{175 (2d-7)^6} \neq 0.
\]
Thus, no tame representations exist when $d \ge 6$.
\end{proof}

Suppose now that we have a wild representation
\begin{equation}\label{E:wild1}
\begin{gathered}
p(x,y) = (a_1x^2 + b_1 x y + c_1y^2)^d + (a_2x^2 + b_2 x y + c_2y^2)^d \\ = 
\sum_{i=0}^{2d} s_i(a_1,b_1,c_1,a_2,b_2,c_2;d) x^{2d-i}y^i,
\end{gathered}
\end{equation}
where $d \ge 4$, $s_{2j+1}(a_1,b_1,c_1,a_2,b_2,c_2;d) = 0$ for $0 \le j \le d-1$, 
$(b_1,b_2) \neq (0,0)$ and \eqref{E:wild1} is not in the form \eqref{E:tame} or 
\eqref{E:tame2}.  

\begin{lemma}\label{KeyLemma}
Suppose  $p\neq 0$ and \eqref{E:wild1} holds. Then, after a scaling of $x$ and $y$, 
\begin{equation}\label{E:wild2}
p(x,y) = p_{\la,\al,\be}(x,y) := (x^2 -\la\al x y + y^2)^d + \la(x^2 + \al x y + \be y^2)^d,
\end{equation}
where $\al\la \neq 0$,  $\be^{d-1} = 1$ and $\la^2 \neq 1$.
\end{lemma}
\begin{proof}
First suppose $b_1 = 0$ in \eqref{E:wild1}. Then $s_1 = da_2^{d-1}b_2$ and $s_{2d-1}= 
db_2c_2^{d-1}$. Since $(b_1,b_2) \neq (0,0)$, we have $a_2=c_2 = 0$ and $p(x,y) = 
(a_1x^2 + c_1y^2)^d + ( b_2 x y )^d$ is even, so $d$ is even and we have
 \eqref{E:tame2}. A similar argument lets us conclude that $b_2 \neq 0$.

Suppose now that $a_1 = 0$. Then  $s_1 = da_2^{d-1}b_2 =0$, and $b_2 \neq 0$ implies
 $a_2 = 0$. It then follows that $y$ divides both $f_3$ and $f_4$, contradicting
Theorem \ref{T:even}. Thus $a_1 \neq 0$, and 
by similar arguments, we have $a_2c_1c_3 \neq 0$. That is, we may assume that all the 
coefficients in \eqref{E:wild1} are non-zero.

We now scale $x$ and $y$ so that $a_1 = c_1 = 1$ and let $\la = a_2^d$, so that, after renaming,
\begin{equation}\label{E:wild1.5}
p(x,y) = (x^2 + \al_1 x y + y^2)^d + \la(x^2 + \al_2 x y + \be y^2)^d,
\end{equation}
where all parameters are non-zero. Returning to the computation, 
\[
\begin{gathered}
s_1 = d(\al_1 + \la \al_2) = 0,\quad 
 s_{2d-1} = d(\al_1  +  \la \al_2\be^{d-1}) = 0.
 \end{gathered}
\]
It follows that $\al_1 = - \la \al_2$, and since $\la\al_2 \neq 0$, it also follows that $
\be^{d-1} = 1$. We now write $\al = \al_2$, so that $\al_1 = -\la \al$, and \eqref{E:wild1.5} 
becomes \eqref{E:wild2}. Finally, if $\la^2 = 1$, then either $\la = 1$ (and \eqref{E:wild2} 
reduces to \eqref{E:tame}), or $\la = -1$ (and \eqref{E:wild2} implies $p=0$).
\end{proof}
\begin{theorem}\label{T:wildanswer}
For $d \ge 4$, the only $\mathcal W_2(4,d)$ set which comes from a wild  representation is
found in \eqref{E:4last2}, and is a  sub-cousin of \eqref{E:quartic13}.
\end{theorem}
\begin{proof}
In view of Lemma \ref{KeyLemma}, we simplify our notation: let
\begin{equation}\label{E:wild22}
\begin{gathered}
p_{\la,\al,\be}(x,y) = \sum_{i=0}^{2d} a_i(\la,\al,\be;d) x^{2d-i}y^i.
 \end{gathered}
\end{equation}
Since $p_{\la,\al,\be}(x,y)$ is even, so is $p_{\la,\al,\be}(y,x)$, as is their difference.
For this reason, write 
\begin{equation}\label{E:wild23}
\begin{gathered}
 \la^{-1}(p_{\la,\al,\be}(x,y) - p_{\la,\al,\be}(y,x) ) 
= (x^2 + \al x y + \be y^2)^d - (\be x^2 + \al x y +  y^2)^d \\
=  \sum_{i=0}^{2d} b_i(\al,\be,d) x^{2d-i}y^i.
\end{gathered}
\end{equation}
We need to find the conditions under which $a_{2j+1}(\la,\al,\be;d)= 0$ for $1 \le 2j+1 \le 2d-1$.
Since $\la b_i(\al,\be) = a_i(\la,\al,\be;d) - a_{2d-i}(\la,\al,\be;d)$ and $\la \neq 0$, it suffices to
consider $a_{2j+1}(\la,\al,\be;d)=b_{2j+1}(\al,\be,d)= 0$ for $1 \le 2j+1 \le d$.

It follows from the definition and $\be^{d-1} = 1$ that
\begin{equation}\label{E:switcheroo}
\begin{gathered}
p_{\la,\al,\be}(x,y) = p_{\la,-\al,\be}(x,-y), \quad 
p_{\la,\al,\be}(x,y) = p_{\la\be,\al/\be,1/\be}(y,x), 
\end{gathered}
\end{equation}
so that, up to linear change,  if $\al^2 = \kappa$ is known, then choosing $\al = \pm\sqrt{\kappa}$   
gives two equations that are cousins. 
Also, any solution for a particular value $\be = \be_0$ will be a cousin of a solution
in which $\be = \be_0^{-1}$. This reduces the number of choices to check.

We now have
\[
\begin{gathered}
a_1(\la,\al,\be) = -d\al\la+d\al\la = 0, \quad
b_1(\al,\be) = d\al(\be^{d-1}-1) =0, \\ 
a_3(\la,\al,\be) = \frac{\la\al d(d-1)}6 \cdot\left( (d-2)\al^2(1-\la^2) + 6(\be-1)\right), \\
b_3(\al,\be) = \frac{\al d(d-1)}6 \cdot (1-\be^{d-3})(6\be + \al^2(d-2)).
\end{gathered}
\]
Now we claim that  $\be \neq 1$ and either
\begin{equation}\label{E:alts1}
\be = -1, \quad \al^2 = \frac{12}{(d-2)(1-\la^2)} \qquad \text{(and $d$ is odd)};
\end{equation}
or
\begin{equation}\label{E:alts24}
\be = \frac 1{\la^2},\quad \al^2 = -\frac 6{\la^2(d-2)}.
\end{equation}
Indeed, since $\al(1-\la^2) \neq 0$, the equation $a_3=0$ implies that $\be \neq 1$ and 
\begin{equation}\label{E:a3}
\al^2 = \frac {6(1-\be)}{(d-2)(1-\la^2)}.
\end{equation}
The equation $b_3=0$ implies that $ (1-\be^{d-3})(6\be + \al^2(d-2)) = 0$. If $\be^{d-3} = 1$,
then $\be^{d-1} = 1$ implies $\be^2 = 1$, and $\be = 1$ is ruled out, so $\be = -1$ and $d$ is odd
and \eqref{E:a3} implies \eqref{E:alts1}. Otherwise, we have by \eqref{E:a3},
\[
0 = 6\be + \al^2(d-2) = 6\be + \frac {6(1-\be)}{(1-\la^2)} = \frac{6(1-\be\la^2)}{1 - \la^2},
\]
so $1 =\be \la^2$ and by \eqref{E:a3},
\[
\al^2 = \frac {6(1-\la^{-2})}{(d-2)(1-\la^2)} = -\frac 6{\la^2(d-2)};
\]
this is summarized as \eqref{E:alts24}.

If $d=4$, then only \eqref{E:alts24} can apply. Since $\be^3 = 1$, $\be \neq 1$ and 
$\om\cdot\om^2 = 1$, we can use \eqref{E:switcheroo} to assume that $\be = \om^2$. 
It follows from \eqref{E:alts24} that
\[
\om^2 = \frac 1{\la^2}, \quad \al^2 =  - \frac 3{\la^2} \implies 
\la = \pm \om^2, \ \al^2 = - 3\om^2.
\]
By \eqref{E:switcheroo}, it suffices to take $\al = \sqrt{-3}\ \om$, but there are two values for 
$\la$: $\la = \pm \om^2$. There are two  wild cases: since $\la\al = \pm \sqrt{-3}$ and 
$(\om^2)^4 = \om^2$,
these are 

\begin{equation}\label{E:wild3}
\begin{gathered}
p_{4,\pm}(x,y):= (x^2  \mp \sqrt{-3}\ xy + y^2)^4 \pm 
\om^2(x^2 +\sqrt{-3}\ \om x y + \om^2 y^2)^4\\= 
 (x^2  \mp \sqrt{-3}\ xy + y^2)^4 \pm 
(\om^2x^2 +\sqrt{-3}\ x y + \om y^2)^4.
\end{gathered}
\end{equation}

We scale the two cases of \eqref{E:wild3} to make them easier to work with. First
\begin{equation}\label{E:w41}
\begin{gathered}
\om^2p_{4,+}(x,\om i y) :=q_1(x,y) = -x^8-14x^6y^2-42x^4y^4-14x^2y^6-y^8 \\
= (\om^2 x^2 - \sqrt{3}\ xy - \om y^2)^4  +(\om x^2 + \sqrt{3}\ xy - \om^2 y^2)^4.
\end{gathered}
\end{equation}
The second line in \eqref{E:w41} is $f_{1,4}^4 + f_{1,5}^4$, which gives a 
new representation after $y \mapsto -y$, namely, $f_{1,3}^4 + f_{1,6}^4$; c.f. \eqref{E:4last1}. 
However, 
the 2-Sylvester matrix of $q_1$ has rank 3, so this case does not fall under Theorem \ref{T:even}.

For the other case, we have 
\begin{equation}\label{E:w42}
\begin{gathered}
- \om^2 p_{4,-}(x,\om i  y) := q_2(x,y) = \\
-( \om^2 x^2 - \sqrt{3} \ xy - \om y^2)^4+  (\om x^2 - \sqrt{3}  \ xy -\om^2 y^2)^4 \\
= \sqrt{-3}\ (x^8 - 14 x^6 y^2 + 14x^2y^6 - y^8).
\end{gathered}
\end{equation}
The 2-Sylvester matrix of $q_2$ has rank 2, so it has a representation as a sum of
two fourth powers. Indeed, \eqref{E:w42} is embedded in \eqref{E:late4}, with two
other representations of $q_2$: one from taking $y \mapsto -y$ in \eqref{E:w42}, and
the other by applying Theorem \ref{T:syl}.

Now suppose $d \ge 5$; more equations need to be satisfied. If \eqref{E:alts1} holds, then 
\[
\begin{gathered}
a_5 = -\frac{8 \sqrt 3 \la(1+\la^2)(d+1)d(d-1)(d-3)}{5((d-2)(1-\la^2))^{3/2}} = 0,
\end{gathered}
\]
so  $\la^2 = -1$, and \eqref{E:alts1} becomes
\begin{equation}\label{E:alts11}
\be = -1, \quad \la^2 = -1, \quad \al^2 = \frac{6}{d-2}.
\end{equation}
If \eqref{E:alts24} holds, then
\begin{equation}\label{E:517}
\begin{gathered}
a_5 = -\frac{ \sqrt 6 (\la^4-1)(2d+1)d(d-1)(d-4)}{10\la^4(d-2)^{3/2}}.
\end{gathered}
\end{equation}
Since $\la^2 \neq 1$, \eqref{E:517} implies $\la^2 = -1$, and simplification yields 
\eqref{E:alts11} again. Observe that $\la = \pm i$ implies that $d \equiv 1 \mod 4$.

If $d=5$, then $\be = -1$, $\la^2 = -1$ and $\al^2 = 2$. 
We choose $\al = \sqrt{2}$ and obtain two  solutions, for $\la = i$ and $\la = -i$, which
we rewrite in terms of the $f_{2,j}$'s, upon noting that $\pm i = (\pm i)^5$:

\begin{equation}\label{E:w5}
\begin{gathered}
p_{5,+}(x,y) = (x^2 -  i\sqrt{2}  xy + y^2)^5 + i(x^2 +\sqrt{2} x y - y^2)^5 = -f_{2,3}^5 -f_{2,4}^5\\ = 
(1+i)(x^{10} + 15 i x^8 y^2 - 30 x^6 y^4 + 30 i x^4y^6-15 x^2y^8 - i y^{10} )
\\
p_{5,-}(x,y) = (x^2 +  i\sqrt{2}  xy + y^2)^5 - i(x^2 +\sqrt{2} x y - y^2)^5 = f_{2,1}^5 + f_{2,4}^5\\=
(1-i)(x^{10} - 15 i x^8 y^2 - 30 x^6 y^4 - 30 i x^4y^6-15 x^2y^8 + i y^{10} )
\end{gathered}
\end{equation}
The expressions in \eqref{E:w5}
 are close cousins; in fact, $p_{5,-}(x,y) = -ip_{5,+}(x,iy)$. Theorem \ref{T:syl} shows that neither 
has a representation as a sum of two even 5th powers; however, $p_{5,-}(x,y) + ip_{5,+}(x,iy) = 0$ 
is a cousin of \eqref{E:quintic11}.

Suppose now that $d \ge 6$; since $d \equiv 1 \mod 4$, we have $d \ge 9$. It turns out
that $b_5=0$ under the conditions of \eqref{E:alts11}, but 
\begin{equation}\label{E:7}
\begin{gathered}
a_7\left(\pm i,\sqrt{\tfrac 6{d-2}},-1,d\right) = 
\pm \frac{8i\sqrt 2(2d-1)(d^3-d)(d-3)(d-5)}{35\sqrt{3}(d-2)^{5/2}}= 0
\end{gathered}
\end{equation}
is clearly impossible for $d \ge 9$, so we are finally done with the wild case.
\end{proof}

\begin{proof}[Proof of Theorems \ref{MT}(8), \ref{W4} and \ref{W5}]
Combine Theorems \ref{T:even}, \ref{T:tameanswer}, and \ref{T:wildanswer}.
\end{proof}

\section{Final remarks}
\subsection{Derivations and historical examples}

It is foolhardy for a living author
to claim priority for any polynomial identity which is verifiable by hand and so
might well have been given as a school algebra assignment. 
We have given previous attributions when we could find them; the pre-1920 literature was scoured 
by Dickson in \cite{D}, but with Diophantine equations over $\nn$ in mind: the coverage of
parameterizations over $\cc$ must be regarded as incomplete. For example, the 1880 paper
\cite{Deb} by Desboves includes both \eqref{E:quartic12} and \eqref{E:quintic11}, and Dickson only 
cites the latter, perhaps because there were no real quintic parameterizations.

Any four binary quadratic forms are linearly dependent, so any $\mathcal W_2(4,d)$ satisfies
both $f_1^d + f_2^d = f_3^d + f_4^d$ and $c_1f_1 + c_2f_2 + c_3f_3+c_4f_4 = 0$ for suitable
$c_i$. It is remarkable that one can find the $\mathcal W_2(4,d)$'s for $d=4,5$ by guessing a 
simple choice of $c_i$'s.

For example, Desboves \cite[p.241]{Deb}  found his version of \eqref{E:quintic11} by assuming 
$f_1 + f_2 = f_3 + f_4$ and $f_1^5 + f_2^5 = f_3^5 + f_4^5$ and parameterizing, to get
\[
0 = (f+g)^5 + (f-g)^5 - ((f + h)^5 + (f-h)^5) = 10f(g^2-h^2)(2f^2 + g^2 + h^2).
\]
He then set $\{f,g,h\} = \{2xy, x^2 -2y^2, i(x^2+2y^2)\}$ via Theorem \ref{W2} and by
scaling via $y \mapsto \sqrt{-1/2}y$, this becomes essentially
\eqref{E:quintic11}. Similarly, after noting that
\[
(f+g)^4 + (f-g)^4 - ((f + h)^4 + (f-h)^4) = 2(g^2-h^2)(6f^2 + g^2 + h^2),
\]
Desboves solved $6f^2 + g^2 + h^2=0$ and derived a cousin  of \eqref{E:quartic12}.

One might also guess  $f_1+f_2+f_3 = 0$;
an old observation (at least back to Proth in 1878 \cite[p.657]{D}) notes that
\begin{equation}\label{E:dio4}
f_1^4 + f_2^4 + (-f_1-f_2)^4 = 2(f_1^2 + f_1f_2+f_2^2)^2,
\end{equation}
so if $f_1^2 + f_1f_2+f_2^2 = g^2$, we obtain a $\mathcal W_2(4,4)$.
Take $f_1= x^2 + y^2$ and  $f_2 = \om x^2 + \om^2 y^2$; this implies
$-(f_1+f_2) = \om^2 x^2 + \om y^2$ and $f_1^2 + f_1f_2+f_2^2 = 3x^2y^2$; hence
$\eqref{E:quartic}$.

In 1904, Ferrari (see \cite[p.654]{D}) gave the ostensibly ternary identity:
\begin{equation}\label{E:fer}
\begin{gathered}
(a-b)^4(a+b+2c)^4 + (b+c)^4(b-c-2a)^4 + (c+a)^4(c-a+2b)^4 \\ 
= 2(a^2 + b^2 + c^2 - ab +ac + bc)^4 
\end{gathered}
\end{equation}
Let $x = a - b$ and $y = b+c$, so that $x+y = a+c$. Then  \eqref{E:fer} becomes 
\eqref{E:quarcous}:
\[
x^4(x+2y)^4 + y^4(-2x-y)^4 + (x+y)^4(y-x)^4 = 2(x^2+xy+y^2)^4.
\]

One can derive \eqref{E:quartic13} by guessing
$(a+d)^4 - (a-d)^4 = (b+d)^4 - (b-d)^4 = (c+d)^4 -(c-d)^4$ for quadratics 
$a,b,c,d $ with $a, b, c$ distinct and  $d\neq 0$. Then routine computations lead to
$a+b+c=0$ and $d^2 = -(a^2+ab+b^2)$. Now set 
$a = x^2 + y^2, b = \om x^2 + \om^2 y^2, c = \om^2 x^2 + \om y^2$, with 
$d^2 = - (a^2 + ab + b^2) = - 3 x^2y^2$, and take $y \mapsto iy$ to get  \eqref{E:quartic13}.

We derived \eqref{E:quintic11} in
\cite[pp.119-120]{Re1} using Newton's Theorem on symmetric polynomials.
  Every symmetric quaternary quintic polynomial $p$ is contained in the ideal 
$\mathcal I = \left( t_1 + t_2+t_3+t_4,  t_1^2 + t_2^2+t_3^2+t_4^2 \right)$.
In particular,  $t_1^5+t_2^5+t_3^5+t_4^5 \in \mathcal I$, so 
\[
f_1+f_2+f_3+f_4 = 0,\  f_1^2+f_2^2+f_3^2+f_4^2 = 0 \implies  f_1^5+f_2^5+f_3^5+f_4^5 = 0.
\]
Upon setting $f_4 = -f_1-f_2-f_3$, the equation $f_1^2+f_2^2+f_3^2+(-f_1-f_2-f_3)^2=0$ can be
analyzed as in Theorem \ref{W2} to obtain \eqref{E:quintic11}.

We present a similar {\it ad hoc, post hoc} derivation for \eqref{E:14}.
\begin{theorem}\label{T:14}
Suppose $S(t_1,\dots,t_6)$ is a symmetric polynomial of degree 7. Then
\[
S\in \mathcal I := \left( \sum_{k=1}^6  t_k, \sum_{k=1}^6  t_k^2, \sum_{k=1}^6  t_k^4 \right).
\]
\end{theorem}
\begin{proof}
Let $e_k$ denote the $k$-th elementary symmetric polynomial. We have $ \sum_{k=1}^6  t_k^2
= e_1^2 - e_2$ and $ \sum_{k=1}^6  t_k^4 = e_1^4 - 4e_1^2e_2 + 2 e_2^2 + 4e_1e_3 - 4e_4$.
Thus, $\mathcal I =  \left(e_1,e_2,e_4\right)$. By Newton's Theorem, $S$ is a linear combination of 
monomials in the $e_k$'s: $e_1^{a_1}e_2^{a_2}e_3^{a_3}e_4^{a_4}e_5^{a_5}e_6^{a_6}$, where
$\sum ka_k = 7$. But 7 cannot be written as a non-negative linear combination of
3, 5 and 6, so each monomial in any such expression
 must contain one of $\{ e_1, e_2,  e_4\}$. 
 \end{proof}

Observe now that if we define $h_j = (\ze_5^{j-1} x^2 + i x y + \ze_5^{-(j-1)} y^2)^2$ for
$1 \le j \le 5$ and $ h_6 = - 5x^2y^2$, then a synching computation shows that 
$\sum_{j=1}^6 h_j = \sum_{j=1}^6 h_j^2 = \sum_{j=1}^6 h_j^4 = 0$. Theorem \ref{T:14} implies that
$\sum_{j=1}^6 h_j^7 = 0$; that is, \eqref{E:14}. 
The mystery now is {\it why} these particular squares work. 

Jordan Ellenberg has
suggested  the following explanation to the author: The surface cut out by $\sum_{j=1}^6 X_j =
\sum_{j=1}^6 X_j^2 = \sum_{j=1}^6 X_j^4$ is a Hilbert modular surface (see \cite[Lemma 2.1]{E1}).
He adds \cite{E2}: ``Dollars to donuts
the nice low-degree rational curve you find on this surface arises as a modular curve on this
modular surface, parametrizing abelian surfaces isogenous to a product of elliptic curves".

\subsection{Representations as a sum of at most two $d$-th powers of quadratic forms}

Which forms $p \in H_{2d}(\cc^2)$ can be written as a sum of two $d$-th
powers of linear forms, and in how many ways? Let 
$A_{d,2}  = \{ (\al_1 x + \be_1 y)^d + (\al_2 x + \be_2 y)^d\}$. 
It is tautological to say that $p \in A_{d,2}$ if and only if there is a linear change taking $p$ into 
$x^d$ or $x^d+y^d$. (A practical test is given by Theorem \ref{T:syl}.)

\begin{corollary}\label{C:62}
If $p \in H_{2d}(\cc^2)$ is not a $d$-th power, then $p$ is a sum of two $d$-th powers of quadratic 
forms if and only if  either (i) $p = \ell^d q$, where $q \in A_{d,2}$, or (ii) after a linear change in 
$p$, $p(x,y) = q(x^2,y^2)$, where $q \in A_{d,2}$. 
\end{corollary}
\begin{proof}
Sufficiency is clear. Conversely, suppose $p = f_1^d + f_2^d$ and $\{f_1,f_2\}$ is honest. As in 
Theorem \ref{T:diag}, there
are two cases. If $\gcd(f_1,f_2) = \ell$ for a linear form $\ell$, then $f_j = \ell \ell_j$, 
giving case (i). Otherwise, we make a linear change which simultaneously diagonalizes $f_1,f_2$,
giving case (ii). 
\end{proof}

If $p$ is a sum of two $d$-th powers in more than one way, then the two representations
together give a $\mathcal W_2(d,4)$. The question is not interesting for $d=2$, since
$p = f^2 + g^2 \iff p = (f+ig)(f-ig)$, so two representations as a sum of two squares amount
to two different factorizations into equal degrees. The situation for $d=3$ is discussed in 
detail in \cite{Re3}; by Theorem \ref{MT}(8), it suffices now to consider $d=4,5$.

If $p$ itself is a $d$-th power, then by Theorem \ref{MT}(3), it does not have another representation 
as a sum of two $d$-th powers. 
In view of Theorems \ref{W4}, \ref{W5}, \ref{T:44}, we have an immediate corollary. 
We choose even representatives (from Theorem \ref{T:even}) and they also happen
to be symmetric (we have taken $y \mapsto \zeta_{16}y$ in \eqref{E:late4}.)
\begin{corollary}\label{C:63}

\

(i) The form $p \in H_8(\cc^2)$ has exactly  two different representations as a sum of two fourth 
powers of binary forms if and only if, after a linear change, it is  
$x^8 + 4 x^6y^2 - 12 x^4y^4 + 4x^2y^6 + y^8$, $x^8-68x^6y^2+6x^4y^4-68x^2y^6+y^8$,
 or
$x^8 -140 x^6y^2 + 294 x^4y^4 -140 x^2y^4 + y^8$.

(ii) The form $p \in H_8(\cc^2)$ has three different representations as a sum of two fourth powers 
of  binary forms if and only if, after a linear change, it is 
$x^8 - 7\sqrt{2}(1+i)x^6y^2 - 7\sqrt{2}(1+i) x^2y^6+y^8$.

(iii) The form $p \in H_{10}(\cc^2)$ has two different representations as a sum of two fifth powers of 
binary forms if and only if, after a linear change, it is $x^{10} - 75 x^8y^2+ 90x^6y^4 + 90x^4y^6 - 
75 x^2y^8 + y^{10}$.
\end{corollary}

\subsection{Open questions}

We have already noted that there exists $k \ge 2$ and $d \ge 6$ so that $\Phi_k(d) >
\Phi_{k+1}(d)$. Gundersen in \cite{Gu} found three meromorphic (not
rational) functions $g_j(t)$ so that $g_1^6 + g_2^6 + g_3^6 = 1$. It is unknown whether
this can be achieved with rational functions. If so, a $\mathcal W_k(4,6)$ set would exist
for some $k>2$.

In case $m = rs$, an $m$-synching on $m$ can be viewed as $r$ coordinated $s$-synchings.
We have not found a useful instance in this when $r=s=2$, although \eqref{E:142} shows
what can happen with $(r,s) = (2,3)$. We hope that improvements on the bounds may
come from careful investigations in this direction.

Another natural question is to restrict our attention to forms with coefficients in a fixed
subfield of $\cc$, such
as $\qq$ or $\rr$. Real forms with even degree also lead to a discussion of ``signatures". 
From the Diophantine point of view,
the equations $A^4 + B^4 + C^4 = D^4$ and $A^4 + B^4 = C^4 + D^4$ are completely
different questions. In this point of view, the real equation \eqref{E:quarcous} is ``(3,1)".  
In 1772, Euler gave a famous (2,2) ``septic" example of a $\mathcal W_7(4,4)$ set
(see \cite[pp.644-646]{D},  \cite[(13.7.11)]{HW},  \cite{L}). 
 So far as we have been able to determine, there are no known real solutions of this kind of
smaller degree, nor proofs that they cannot exist. 

Theorem \ref{W2} shows that \eqref{E:pyt} is ``universal" in presenting all $\mathcal W_k(3,2)$
sets; that is, projectively, all families come from the substitution $(x,y)\mapsto (g,h)$. Are the 
solutions given in Theorems \ref{W3}, \ref{W4} and \ref{W5} also universal in
this sense? The answers are ``no" for $d=3,4$. These families are all linearly dependent. 
For $d=3$, the family in \eqref{E:quarcube} is linearly independent, as are the 
parameterizations of the Euler-Binet solutions 
to $x^3 + y^3 = u^3 + v^3$ (see e.g \cite[(13.7.8)]{HW}), when viewed as elements of $\cc[a,b,\la]$.
 For $d=4$, it can be checked that the Euler septics are also linearly independent.
The case $d=5$ is open. Can the sets $\mathcal W_k(4,d)$ themselves 
be parameterized for $k \ge 3$?

 Finally, we note that the intricate calculations of section three and four suggest that new methods 
 will be needed to study $\mathcal W_k(r,d)$ for $r > 4$ or $k > 2$. Nevertheless, we make the
 following conjecture, based on Theorem \ref{MT}:
 
 \begin{conjecture}\label{C:end}
 There is a small constant $M$ so that, for all $k$ and $d$,
 \[
 \left\vert \Phi_k(d) -  \min_{1 \le i \le k}\left(\frac di + i \right) \right\vert < M.
 \]
 \end{conjecture}
 

\end{document}